\documentclass[12pt]{article}
\usepackage[a4paper]{geometry}
\usepackage{amsthm}

\usepackage{amssymb}
\usepackage{amsmath}
\usepackage{eufrak}

\newtheorem{theorem}{Theorem}

\newtheorem{corollary}[theorem]{Corollary}
\newtheorem{lemma}[theorem]{Lemma}

\newtheorem{proposition}[theorem]{Proposition}

\begin{document}
\def\F{{\mathbb F}}
\title{ From pre-Lie rings back  to braces}
\author{  Agata Smoktunowicz}
\date{ }
\maketitle

\begin{abstract} 

 Let $A$ be a brace of cardinality $p^{n}$ where $p>n+1$ is prime, and  $ann (p^{i})$ be the set of elements of additive order at most $p^{i}$ in this brace.   A pre-Lie ring  related to the brace $A/ann(p^{2})$ was constructed  in \cite{shsm}.

 We show   that there is a formula dependent only  on the additive group of the brace $A$ which reverses the construction from \cite{shsm}. As an  application example it is shown that 
   the brace $A/ann(p^{4})$ is the group of flows of a left nilpotent pre-Lie algebra. 

  %
  \end{abstract}

\section{Introduction}
  Let $A$ be a brace of cardinality $p^{n}$ where $p>n+1$ is prime, and  $ann (p^{i})$ be the set of elements of additive order at most $p^{4}$ in this brace. 
 A pre-Lie ring  related to the brace $A/ann(p^{2})$ was constructed  in \cite{shsm}.
   The aim of this paper is to show that this construction can be reversed to recover the brace $A/ann(p^{2})$ by applying a formula similar to the formula for the group of flows to the obtained pre-Lie ring. This formula is the same for braces which have the same additive group.  
   The main result  of this paper is the following application of this result:   
  \begin{theorem}\label{Main}  Let $p$ be a prime number, and $n<p-1$ be a natural number. Let $(A, +, \circ )$ be a brace of cardinality $p^{n}$, and let $ann (p^{4})$ be the set of elements 
  of additive order at most $p^{4}$ in this brace.
   Then the multiplication in the brace $A/ann(p^{4})$ is the same as the multiplication in the  group of flows of some left nilpotent pre-Lie algebra $(A/ann(p^{4}), +, \cdot )$.
   Moreover, the addition in the pre-Lie algebra $(A/ann(p^{4}), +, \cdot )$ and in the  brace $(A/ann(p^{4}), +, \circ )$ is the same (where $(A/ann(p^{4}), +, \circ )$ is the factor brace of the brace $(A, +, \circ )$ by the ideal $ann(p^{4})$).
  \end{theorem} 

    Note that, by Lemma $17$,  \cite{passage}, $ann(p^{i})$ is an ideal in the above brace $A$.
     Therefore the brace $A/ann(p^{i})$ is well defined. 
   Recall that the passage from finite pre-Lie algebras to finite braces was first  discovered by Wolfgang Rump in \cite{Rump}.  He used the  exponential function $e^{a}$ to pre-Lie algebras and showed that the obtained structure is a brace. 
    This construction can also be described using the group of flows developed in \cite{AG}. For more details, see \cite{passage}.
 It is an open question as to whether or not every brace with such cardinality 
 could be be obtained from some  pre-Lie algebra in this way. This is  known to be true for right nilpotent braces for sufficiently large $p$ \cite{passage}, and it is also  known to be true for $\mathbb R$-braces \cite{Rump}, where for $\mathbb R$-braces the correspondence is local.
   
   In \cite{Rump}, page 141,  Wolfgang  Rump  suggested a potential way of approaching this question using $1$ cocycles. However, there are complications, because the additive group of a Lie algebra and the additive group of the brace may not be identical in the case when the adjoint group of the brace is obtained by using the Lazard's correspondence from this Lie algebra.

Notice that Theorem \ref{789} implies that if $A$ is a brace of cardinality $p^{n}$ for a prime number $p$ and a natural number $n<p-1$ then the factor brace  $A/ann(p^{2})$ is obtained by the formula 
  from Theorem \ref{789} from some pre-Lie ring with the same additive group. It is not clear whether the formula from  Theorem \ref{789}
 is the same as the formula for the group of flows. We will use the same notation as in \cite{shsm}. 
 For the background section we refer the reader to the background section of \cite{shsm}.
\section{ Left nilpotency of the  pre-Lie rings constructed in \cite{shsm} }
   Let $(A, +, \circ )$ be a brace. One of the mappings used in connection with braces are the  maps $\lambda _{a}:A\rightarrow A$ for $a\in A$. Recall that  for $a,b, c\in A$ 
   we have $\lambda _{a}(b) = a \circ  b - b$, $\lambda _{a\circ c}(b) = \lambda _{a}(\lambda _{c}(b))$. Our first result shows that the pre-Lie rings constructed in \cite{shsm} are left nilpotent. For the reader's convenience 
    the main result and the construction from \cite{shsm} is quoted in Theorem \ref{1} at the end of this paper.
   
   \begin{lemma}\label{2} Let the notation be as in Theorem \ref{1}.  Then the obtained pre-Lie ring $(A, +, \bullet )$  is left nilpotent.
  \end{lemma} 
  \begin{proof} By Lemma  \ref{citepassage},  we have $pA = A^{\circ p}$. Consequently, if $a\in A$ then 
  \[pa=a_{1}^{\circ p} \circ  a_{2}^{\circ p} \circ \cdots \circ a_{m}^{\circ p}\]
for some $m$ and some $a_{1}, \cdots , a_{m}\in A$. Recall that $\circ $ is an associative operation, so we do not need to put brackets in this formula.
By Lemma \ref{citenote},  we have that 
  \[a^{\circ p} *b =  \sum_{k=1}^{p-1}{p\choose k}e_{k},\] where 
  $e_{1} = a*b$, $e_{2} = a*(e_{1})$,
 and for each $i$,  $e_{i+1} =a*e_{i}$.
  
   Notice that if $b\in A^{i}$  then $e_{j} \in A^{i+j}$, hence $e_{p-1}\in A^{p-1}\subseteq A^{n+1}=0$.
   
By using this formula, we see that if  $e\in A$, $b\in A^{i}$ and $c\in A^{i+1}$ for some $i$ then
   \[ \lambda _{e^{\circ p}}(pc+b)=e^{\circ p}*(pc+b)+pc+b=pc' +b,\] 
for some $c'\in A^{i+1}$ (since ${p\choose k}$ is divisible by $p$ for $0<k<p$).  
  Recall that $A^{i+1}=A*A^{i}$. Let $b\in A^{i}$ for some $i$.  Notice that by the multiplicative property of $\lambda $ maps we have:
\[(pa)*b+b=\lambda _{pa}(b)=\lambda _{a_{1}^{\circ p}\circ \cdots \circ a_{m}^{\circ p}}(b) =
\lambda _{a_{1}^{\circ p}}(\cdots (\lambda _{a_{m-1}^{\circ p}} (\lambda _{a_{m}^{\circ p}}(b)\cdots )))\subseteq pA^{i+1}+b.\] 
     This implies
\[(pa)*b+b\in pA^{i+1} +b,\]
hence $[\wp^{-1}((pa) * b)] \in [A^{i+1}]$, provided that $b\in  A^{i}$, where $[A^{i+1}]=\{[a]:a\in A^{i+1}\}$, and, as usual, we use notation $[a]=[a]_{ann(p^{2})}$.
       Therefore, by the formula for operation $\bullet $ from Theorem \ref{1}, we get 
 \[[a]\bullet [b] \in [A^{i+1}],\] for  $b\in  A^{i}$, $a\in A$. Consequently,
\[[b_{1}] \bullet  ([b_{2}]\bullet ( \cdots  ([b_{n}] \bullet  [b_{n+1}]) ))\in [A^{n+1}]= 0,\] for all $b_{1}, . . . b_{n+1}\in  A$, since  $A^{n+1}=0$ in every brace of cardinality $p^{n}$ for a prime $p$, by a result of Rump \cite{rump}.
  \end{proof} 
  
\section{ Relations between  $\odot $ and  $\bullet $}

 Let $(A, +, \circ )$ be a brace, and let $a\in A$. In this section, by $[a]$, we mean the coset of $A$ in the factor brace $A/ann(p^{2})$, so 
 $[a]=\{a+i: i\in ann (p^{2})\}$. Recall that $[a]$ is an element of the factor brace  $A/ann(p^{2})$. 
  We will denote the multiplication and the addition in the brace  $A/ann(p^{2})$ by  the same symbols as the addition and the multiplication in brace $A$. Recall that $[a]+[b]=[a+b]$ and $[a]*[b]=[a*b]$, $[a\circ b]=[a*b+a+b]$. 
 
  We recall a lemma from \cite{shsm}. 
\begin{lemma}\label{6}
 Let $(A, +, \circ )$ be a brace of cardinality $p^{n}$, $p$ is a prime number, and $p>n+1$. Let $\wp^{-1}:pA\rightarrow A$ be defined as in \cite{shsm}. Let $a,b\in A$,  then $[a]=[a]_{ann(p^{2})}$, $[b]=[b]_{ann(p^{2})}$ are elements of $A/ann(p^{2})$.  
  Define 
  \[[a]\odot [b]=[\wp^{-1}((pa)*b)],\]
   then this is a well defined binary operation on $A/ann(p^{2})$.
\end{lemma}
 Notice that $\wp^{-1}$ can be defined in the same way for all braces whose additive group is the same group $(A, +)$. Recall also that if $a=px$ then $\wp^{-1} (a)$ is an element in $A$ such that $p(\wp^{-1}(x))=a$.
  
 
   
   
{\em Remark $1$.}  Let $(A, +, \circ )$ be a brace of cardinality $p^{n}$ for some prime number $p$, and for some natural number $n<p-1$.  
   Note that  in the  proof of the following  Lemma \ref{F} only the following facts are used:
     \begin{itemize}
     \item $p^{p-1}a=0$ for all $a\in A$.
     \item The product of any $p-1$ elements from $pA$, under the operation $*$,  is zero (by Proposition \ref{b}).
     \item For every $i$,  $p^{i}A$ is an ideal in $A$, by \cite{Engel}. 
     \item  The formula from Lemma \ref{co}.
     \item Operation  $\wp ^{-1}:pA\rightarrow A$, which only depends on the additive group, and it can be assumed that for all braces with the same additive group this operation $\wp ^{-1}$ is the same.
     \item The operation of taking the coset $[a]=[a]_{ann(p^{2})}$ of an element $a\in A$. This coset only depends on the additive group of this brace and does not depend on the multiplicative group of this brace.
     \item The inductive assumption, which gives formulas which only depend on the additive group of brace $A$.
     \end{itemize}    
   $ $
   
   {\bf Definition 1.}  Let $V_{j}$ denote the set of all non-associative words in non-commuting variables $X_{1},\ldots , X_{j-1},X_{j}$ (where $X_{j}$ appears only once at the end of each word and all words have length larger than $2$). Let $w$ be a non-associative  word in variables $X_{1}, \ldots X_{j}$, so $w\in V_{j}$ 
   Let $(A, +, \circ )$ be a brace, and let $x_{1}, \ldots , x_{j}\in A$.
    Let $w\langle [x_{1}], \ldots , [x_{j}]\rangle $
    be the specialisation  of $w$, for $X_{i}=[x_{i}]$ and under the operation $\odot $, and
     let $w\{x_{1}, \ldots , x_{j}\}$ be the specialisation  of $w$, for $X_{i}=x_{i}$ and under the operation $*$.
    For example, let $w=(X_{1}X_{2})X_{3}$, and let $a,b\in A$, then $w\langle [a], [a], [b]\rangle =([a]\odot [a])\odot [b]$ and 
$w{a,a,b}=(a*a)*b$.  
   We will use this notation in the following lemma.

\begin{lemma} \label{F}   Let $p$ be a prime number and $n<p-1$ be a natural number. Let $(A, +)$ be an abelian group of cardinality $p^{n}$, where $p$ is a prime number and $n<p-1$ is a natural number.  Let $j>2$ and  let $W\in V_{j}$  be a non-associative  word in variables $X_{1}, \ldots X_{j}$, where each $X_{i}$ appears only once, and
    they appear in the order $X_{1}, \ldots, X_{j}$. 
 Let $i_{1}, \ldots i_{j-1}>0$, where $i_{j}\geq 0$ are all natural numbers.
  There are integers $\beta _{v}$ for $v\in V_{j}$, with only a finite number of these integers  nonzero, and such that
 for each brace $(A, +, \circ )$  with the additive group $(A, +)$   and for each $x_{1}, \ldots , x_{j}$ we have 
\[[\wp^{-1}W\{p^{i_{1}}x_{1}, \ldots , p^{i_{j}}x_{j}\}]=   p^{i_{1}+\ldots +i_{j}-1}W\langle [x_{i}],  \ldots , [x_{j}] \rangle +p^{i_{1}+\ldots +i_{j}}\sum_{v\in V_{j}}\beta _{v}v\langle [x_{1}], \ldots , [x_{j}]   \rangle ,\]
\end{lemma}
\begin{proof} Denote 
 \[w=W\{p^{i_{1}}x_{1}, \ldots , p^{i_{j}}x_{j}\}.\]  We proceed on induction, in the decreasing order,  on  
 \[i_{1}+i_{2}+\cdots +i_{j}.\]
  
 If $i_{1}+\cdots +i_{j}\geq p-1\geq n+1$ then $w$ is zero by Proposition \ref{b}. 
 Then \[[\wp^{-1}(w)]=[\wp^{-1}(0)]=[0]=p^{i_{1}+\cdots +i_{j}-1}[\bar  w]\]   since $[p^{p-2}A]=[0]$  so 
 $p^{i_{1}+\ldots +i_{j}-1}[\bar w]=[0].$ Therefore, the result is true.

 Let $k$ be a natural number. Suppose now that the result holds in cases when $power(w)>k$, and for all $j$  and we will show that the result holds also in the case when $power(w)=k$ and for all $j$. 
 For this fixed $power(w)=k$, we will proceed by induction on $j$ (in the  usual increasing order). 
The smallest possible $j$ to consider is $j=2$. 
 Recall that $w$ is a product of elements $p^{i_{1}}x_{1}, \ldots , p^{i_{j}}x_{j}\in A$ under the operation $*$. 
 We proceed by  induction on $j$.

{\bf Case $j=2$.} Notice that for $j=2$ we have $w=(p^{k'}x_{1})*(p^{m}x_{2})$, where $k'+m=k$.  
If $k'=1$ then $w=(px_{1})*x_{2}$.  Consequently \[[\wp^{-1}(w)]=[\wp^{-1}((px_{1})*(p^{m}x_{2})] =
p^{m}[\wp^{-1}((px_{1})*x_{2})]=p^{m}[x_{1}\odot x_{2}]=p^{k-1}[x_{i}\odot x_{2}].\]

 It remains to consider the case $i_{1}=k'>1$, $i_{2}=m$.
 
 By Lemma \ref{citenote}: 
        \[ (px_{1})^{\circ  {k'-1}}=p^{k'}x_{1}+\sum_{i=2}^{p-1}{{p^{k'-1}\choose i}}y_{i},\] 
  where $y_{2}=(px_{1})*(px_{1}), \ldots $ and  $y_{i+1}=(px_{1})*y_{i}$ for $i=2,3, \ldots , p-1$. 
  Consequently \[ (px_{1})^{\circ  {k'-1}}=p^{k'}x_{1}+\sum_{i=2}^{p-1}\zeta _{i}p^{k'-1}y_{i},\]
   for some integers $\zeta _{i}\geq 0$. 
   This follows because  $-x=(p^{p}-1)x=x+\cdots +x$ for $x\in A$.
   Notice that the sum ends at the $p-1$-th place  since $y_{i}\in A^{i}$ and $A^{n+1}=0$ (since $A$ has cardinality $p^{n}$ by \cite{rump}, and $n+1\leq p-1$, hence $A^{p-1}=0$).
       Therefore  
 \[(p^{k'}x_{1} )*(p^{m}x_{2})= ((px_{1})^{\circ p^{k'-1}}+\sum_{i=2}^{p-1}\zeta _{i}p^{k'-1}y_{i})*(p^{m}x_{2}).\]
  By Lemma \ref{co} and Corollary \ref{777} we have 
 \[(p^{k'}x_{1} )*x_{i_{2}}= (px_{1})^{\circ p^{k'-1}}*(p^{m}x_{2})+d\] where
  $d$ is a sum of some products of some number of copies of elements $p^{k'-1}y_{i}$  for $i=2,3, \ldots $ and element $p^{m}x_{2}$ at the end, and also possibly  some number of copies of element $(px_{1})^{\circ p^{k'-1}}$.
   Observe that $p^{k-1}y_{i}$ is a product of $i-1$ copies of element $px_{1}$ and element $p^{k'-1}(px_{1})=p^{k'}x_{1}$ at the end, so $p^{k'-1}y_{i+1}=(px_{1})*((px_{1})*(\cdots *( (px_{1})*(p^{k'}x))\cdots )$,  hence in this presentation $p$ appears at least $k'+1$ times near elements $x_{i}$ in this product which is equal to  $p^{k'-1}y_{i}$ for $i=2, 3, \ldots $.  

  Then, by applying  Lemma \ref{co} several times to the element $(px_{1})^{\circ p^{k'-1}}=p^{k'}x_{1}+\sum_{i=2}^{p-1}p^{k'-1}\zeta _{i}y_{i}$ which appears in these products we get that $d$ is a  sum of some products of elements $p^{k'-1}y_{i}$  and the element $p^{m}x_{2}$ at the end, and possibly also some other elements (which are products of $y_{i}\in pA$ and $p^{k'}x_{1}$). Notice that the process of applying Lemma \ref{co} will terminate at some stage,  because we will obtain longer products of elements from $pA$, and by Proposition \ref{b} such products of more than $p-1$ elements are zero. Notice also  that   $p$ appears at least $k'+1+m=k+1$ times near elements $x_{i}$  in  such products.

  Recall that $[\wp(x+y)]=[\wp^{-1}(x)]+[\wp^{-}(y)]$ for $x,y\in pA$, by \cite{shsm}. 
 Observe now that   
\[[\wp^{-1}((px_{1})^{\circ p^{k'-1}}*(p^{m}x_{2}))]=
[\wp^{-1}(p^{k'-1} ((px_{1})*(p^{m}x_{2})))]+[\wp^{-1}(c)],\]
where \[c=\sum_{i=2}^{p-1}\zeta_{i}\wp^{-1}(p^{k'-1}z_{i}),\]
 where $z_{2}=(px_{1})*((px_{1})*(p^{m}x_{2}))$ and for $i=2, 3, \ldots $ we have  $z_{i+1}=(px_{1})*z_{i}$ (by Lemma \ref{citenote}).
  Notice that  $p^{k'-1}z_{2}=(px_{1})*((px_{1})*(p^{m+k'-1}x_{2}))$ and $p^{k'-1}z_{i+1}=(px_{1})*(p^{k'-1}z_{i})$. 
 Therefore, in this presentation of elements $p^{k'-1}z_{i}$, $p$ appears at least $m'+k'+1$ times near elements $x_{i}$.  
 
Notice also that 
\[[\wp^{-1}(p^{k'-1} ((px_{1})*(p^{m}x_{2})))]=p^{k'-1+m}[\wp^{-1}((px_{1})*x_{2})]=p^{k-1}[x_{1}]\odot [x_{2}].\]

 Combining it all together, we obtain 
 
 \[[\wp^{-1}((p^{k'}x_{1})*(p^{m}x_{2}))]=p^{k-1}[x_{1}]\odot [x_{2}]+[\wp^{-1}(d)]+ [\wp^{-1}(c)].\]
  Observe that this equation was proved without using the inductive assumption, so it is true for all $k'>0, m\geq 0$.
 
  Recall  that $k'+m=k$. 
 Remark $1$ and the inductive assumption on $k$ (applied for numbers larger than $k$) shows that $[\wp^{-1}(d)]+ [\wp^{-1}(c)]$ can be written as some sum   \[p^{k}\sum_{v\in V_{j}}\beta _{v, i_{1}, \ldots , i_{j},w}v\langle [x_{1}], \ldots , [x_{j}]   \rangle \]

  This proves the case $j=2$. 
  
{\bf Case  $j>2$.} Suppose now that $j>2$.
 Recall that $w$ is a product of elements $p^{i_{1}}x_{1}, \ldots , p^{i_{j}}x_{j}\in A$ under the operation $*$. 
  Notice that there is $t$ such that $w$ is a product of elements $p^{i_{1}}x_{1}, \ldots  p^{i_{t-1}}x_{t-1}$ and element  $p^{i_{t}}x_{t}*p^{i_{t+1}}x_{t+1}$ and elements $p^{i_{t+2}}x_{t+2},\ldots ,  p^{i_{j}}x_{j}\in A. $
   
Recall that several lines prior we proved that \[[\wp^{-1}((p^{k'}x_{1})*(p^{m}x_{2}))]=p^{k-1}[x_{1}]\odot [x_{2}]+[\wp^{-1}(d)]+ [\wp^{-1}(c)].\]
  Observe that this equation was proved without the use of inductive assumption, so it is true for all $k'>0,m\geq 0$.
  We can apply it to $p^{i_{t}}x_{t}$ instead of $p^{k'}x_{1}$ and to $p^{i_{t+1}}x_{t+1}$ instead of 
 $p^{m}x_{2}$, and obtain that: 
 \[[\wp^{-1}((p^{i_{t}}x_{t})*(p^{i_{t+1}}x_{t+1}))]=p^{i_{1}+i_{2}-1}[x_{t}]\odot [x_{t+1}]+[\wp^{-1}(d')]+ [\wp^{-1}(c')],\]
 for some $d'$, $c'$ (which were calculated in the same way as $d$ and $c$ above were calculated). 
  Observe that 
  \[c'+d'=\sum_{i}r_{i}\] where each $r_{i}$  is a product of some copies of elements from the set $\{p^{l}x_{t}, p^{l}x_{t+1}:l=1,2, \ldots \}$ and $p$ appears more than  $i_{t}+i_{t+1}$ times near elements $x_{t}, x_{t+1}$ in each product (which is equal  $r_{i}$). 
 
 Therefore, by multiplying the above equation by $p$ we get  
 \[(p^{i_{t}}x_{t})*(p^{i_{t+1}}x_{t+1})=p^{i_{t}+i_{t+1}-1}h+c'+d' +a,\]
 for some $h\in A$ such that $[h]=[x_{1}]\odot [x_{2}]$ and for some $a\in A$. Moreover $a\in ann(p)=\{x\in A:px=0\},$ since $a$ is an element from $ann(p^{2})$ multiplied by $p$. 
 
  By Lemma \ref{co} and Corollary \ref{777} we have that $w=w'+\sum_{i}w_{i}+a'$ where $w'$ (and respectively $w_{i}$ and $a'$) is a product of elements 
 $p^{i_{1}}x_{1}, \ldots  p^{i_{t-1}}x_{t-1}$ and element  $p^{i_{t}+i_{t+1}}h$ (and respectively element $r_{i}$ and $a'$) and elements $p^{i_{t+2}}x_{t+2},\ldots ,  p^{i_{j}}x_{j}\in A$. 
Notice that $a'$ belongs to $ann(p)$, since $ann(p)$ is an ideal in $A$ by \cite{passage}.

Notice that  $w'$ is a product of less than $j$ elements (which are $[x_{i}]$ for $i\neq t$, $i\neq t+1$ and the element $[x_{t}\odot x_{t+1}]$), and  $p$ appears $k$ times in this product (near elements $[x_{i}]$ for $i\neq t$, $i\neq t+1$ and near  the element $[x_{t}\odot x_{t+1}]$),   so we can apply the inductive assumption on $j$ to $w'$ to get 
    \[[\wp^{-1}(w')]=p^{k-1}w\langle [x_{1}], \ldots , [x_{j}]\rangle +p^{k}\sum_{v\in V_{j}}\beta _{v}'v\langle [x_{1}], \ldots , [x_{j}]   \rangle,\]
    for some integers $\beta _{v}'$.  

On the other hand, observe that $p$ appears more than $k$ times when we present each $w_{i}$ as a product of elements 
$p_{l}x_{i}, p_{l}[x_{t}]\odot [x_{t+1}]$ for some $i, l$ (near elements $x_{i}$ and near the element $[x_{t}]\odot [x_{t+1}]$), hence we can apply the inductive assumption on $k$ (applied to numbers larger than $k$) to elements $w_{i}$).  We get that  
    \[[\wp^{-1}(w_{i})]=p^{k}\sum_{v\in V_{j}}\beta _{v}''v\langle [x_{1}], \ldots , [x_{j}]   \rangle \] 
    for some integers $\beta _{v}'$.  
    Notice also that $\wp^{-1}(a)\in ann(p^{2})$ since $a'\in ann(p)$, hence $[\wp^{-1}(a)]=[0]$.
 Observe that we  implicitly used Remark $1$. When combined together  this concludes the proof.
\end{proof}  
 The aim of this section is to prove the following result:
\begin{theorem}\label{11}  Let $p$ be a prime number and $n<p-1$ be a natural number. Let $(A, +)$ be an abelian group of cardinality $p^{n}$.  Let $V$ denote the set of all non-associative words in non-commuting variables $X,Y$ (where $Y$ appears only once at the end of each word and $X$ appears at least twice  in each word in $V$). Then there are integers $\alpha _{w}$, for $w\in V$, such that only a finite number of them is non-zero and that the following holds: 
 For each brace $(A, +, \circ )$  with the additive group $(A, +)$ and for each $a, c\in A$ 
 we have 
\[[a]\odot   [c]= a\bullet  c +p\sum_{w\in V}\alpha _{w}w([a],[c]),\]  where  $w([a],[c])$ is a specialisation of the word $w$ for $X=[a]$, $Y=[c]$, and the multiplication in $w(a,c)$ is the same as the  multiplication in the pre-Lie ring $(A/ann (p^{2}), +, \bullet )$ constructed in Theorem \ref{1} from the brace $(A, +, \circ )$. So for example if $w=((XX)X)Y$ then $w(a,c)=(([a]\bullet [a])\bullet [a])\bullet [c])$.  
\end{theorem}
\begin{proof}
   Observe that, by applying Corollary \ref{777} and Lemma \ref{co} several times (and using the fact that $p^{p-1}A=0$)  we can write each element   $[(\xi^{i} pa)*c]$ as a sum of  the element $\xi^{i}[ (pa)*c]$  and  $t(pa,c,i)$, where $t(pa,c, i)$ is  a sum of products of elements $pa$ and element $c$ at the end, with $ap$ appearing at least two times in each product. Consequently 
    \[[\wp^{-1}([(\xi^{i} pa)*c)]=\xi^{i} [\wp^{-1}((pa)*c)]+[\wp^{-1}(t(pa,c,i))].\]
   By Lemma \ref{F} applied to the products which are  summands of  $t(pa, c,i)$ we have: 
   \[[\wp^{-1}([(\xi^{i} pa)*c)]=\xi ^{i}[a]\odot [c]+ p[f(pa,c, i)]\]
   where $f(pa,c,i)$  is a sum of products of  at least two copies of element $a$ and element $b$ appearing only once at the end, under operation $\odot $. Recall the formula for $[a]\bullet [c]$ from Theorem \ref{1}.   By applying this formula we get  \[[a]\bullet  [c]= (p-1) [a]\odot [c]+p\sum_{i=0}^{p-2}\xi ^{i}f(a,c, i).\]    
   
   Notice that $\sum_{i=0}^{p-2}\xi ^{i}f(a,c, i)$ is a sum of some products of $[a]$ and $[c]$ under operation $\odot $ and by Remark $1$ the types of these products do not depend on the particular elements $[a], [c]$ which were used  (they only depend on the additive group of $A$), moreover  a  product may appear several times in this sum.
   
    Therefore
   \[[\wp^{-1}([(\xi^{i} pa)*c)]=\xi ^{i}[a]\odot [c]+ p\sum_{w\in V}\zeta _{w,i} w\langle [a],[c]\rangle \]
   for some $\zeta _{w}$ which don't depend on $a$ and $c$ (and $\zeta _{w}$ are the same in all braces with the same additive group $(A, +)$),
   where  $w\langle [a],[c]\rangle $ is a specialisation of word $w\in V$ for $X=[a]$, $Y=[c]$, and the
    multiplication in $w\langle [a],[c]\rangle $ is $\odot $. For example in $w=(XX)Y$ then $w\langle [a], [c]\rangle =([a]\odot [a])\odot [c]$. 
   Consequently 
   \[[a]\bullet [c]=(p-1)[a]\odot [c]+p\sum_{w\in V}m _{w} w\langle [a], [c]\rangle ,\]
   where  $m_{w}=\sum_{i=0}^{p-2}\xi ^{p-1-i}\zeta _{w,i}$.

  {\bf Part 1.} Let $x_{1}, \ldots , x_{j}\in A$. 
   Let $u$ be a non-associative  word in variables $X_{1}, \ldots X_{j}$ (and each $X_{i}$ appears only once, and
    they appear in the order $X_{1}, \ldots, X_{j}$.) Let $\langle [x_{1}], \ldots , [x_{j}]\rangle $
    be the specialisation  of $u$ for $X_{i}=[x_{i}]$ and under the operation $\odot $, and $u([x_{1}], \ldots , [x_{j}])$ be the specialisation  of $u$ for $X_{i}=[x_{i}]$ and under the operation $\bullet $.
    
  We will  now prove by induction on $j$, that 
  if $[x_{1}], \ldots , [x_{j-1}] \in S_{1}, [x_{j}]\in S_{2}$ then there are integers $m_{w,u}$, such that 
   
  \[u([x_{1}], \ldots  , [x_{j}])=(p-1)^{j -1}u\langle [x_{1}], \ldots  , [x_{j}]\rangle+p\sum_{w\in W}m _{w,u} w\langle [x_{1}], \ldots  , [x_{j}]\rangle .\]
   Moreover, $W$, the set of non-associative words of length at least $3$, and  in $j$  variables $X_{1}, \ldots, X_{j}$ for some $j$, and 
   $w([x_{1}], \ldots , [x_{j}])$, is a specialisation of $w\in W$ for $X_{i}=[x_{i}]$ under the operation $\odot $.

   For $j=2$ the result follows from the first part of our proof, since we have shown that
\[[a]\bullet [c]=(p-1)[a]\odot [c]+p\sum_{w\in W}m _{w} w\langle [a], [c]\rangle .\]   
    Therefore,  for $u=X_{1}X_{2}$ we take  $m_{v, u}=m_{v}$.

    We proceed by induction on $j$. Let $j>2$. There is $1<t<j$ such that 
  $u=vy$ for some word $v$ in variables  $X_{1}, \ldots , X_{t}$, and some word $y$ in variables $X_{t+1}, \ldots , X_{j}$. 
   By the inductive assumption: 
   \[v([x_{1}], \ldots , [x_{t}])=(p-1)^{t-1}\langle [x_{1}], \ldots , [x_{t}]\rangle +p\sum_{w\in W}m_{w,v}w\langle [x_{1}], \ldots , [x_{t}]\rangle \]
   and 
  \[y([x_{t+1}], \ldots , [x_{j}])=(p-1)^{j-t-1}y\langle [x_{t+1}], \ldots , [x_{j}]\rangle +p\sum_{w\in W}m_{w,y}w\langle [x_{t+1}], \ldots , [x_{j}]\rangle .\]

  It follows that,   
 \[u([x_{1}], \ldots , [x_{j}])=v([x_{1}], \ldots , [x_{t}])\bullet  y([x_{t+1}], \ldots , [x_{j}])=\]
 \[=(p-1)^{t-1}v\langle [x_{1}], \ldots , [x_{t}]\rangle +p\sum_{w\in W}m_{w,v}w\langle [x_{1}], \ldots , [x_{t}]\rangle \bullet z\] 
 where \[z= (p-1)^{j-t-1}y\langle [x_{t+1}], \ldots , [x_{j}]\rangle +p\sum_{w\in W}m_{w,y}w\langle [x_{t+1}], \ldots , [x_{j}]\rangle .\]
 
  Recall that operation $\bullet $ is distributive with respect to addition, since $\bullet $ is a pre-Lie operation. 
  Therefore  
 \[u([x_{1}], \ldots , [x_{j}])=(p-1)^{j-2}v\langle [x_{1}], \ldots , [x_{t}]\rangle \bullet  y\langle [x_{t+1}], \ldots  , [x_{j}]\rangle + \]
 \[+p\sum_{w, w'\in W}m_{w,w',v,y}w\langle [x_{1}], \ldots , [x_{t}]\rangle \bullet w'\langle [x_{t+1}], \ldots , [x_{j}]\rangle ,\]
 for some integers $m_{w,w', v, y}$. Moreover, in the summation the word $w$ only depends on variables $X_{1}, \ldots , X_{j}$, and the word $w'$ depends only on variables $X_{t+1}, \ldots , X_{j}$.
  Therefore, 
   \[w\langle [x_{1}], \ldots , [x_{t}]\rangle \odot  w'\langle [x_{t+1}], \ldots , [x_{j}]\rangle=w''\langle [x_{1}], \ldots , [x_{j}]\rangle \] where $w''=ww'$ (so $w'$ is the word which is obtained by putting the word $w'$  after $w$, it could be also written as $w''=(w)(w')$).

    Observe also that \[v([x_{1}], \ldots, [x_{t}])\odot  y\langle [x_{t+1}], \ldots , [x_{j}]\rangle =u\langle [x_{1}], \ldots , [x_{j}]\rangle ,\]
    since $u=vy$.
   
   The result now follows from the formula 
   \[[a]\bullet [c]=(p-1)[a]\odot [c]+p\sum_{w\in W}m _{w} w\langle [a], [c]\rangle ,\]
   applied for $[a]=v\langle [x_{1}], \ldots , [x_{j}]\rangle $ and $[c]=y\langle [x_{1}], \ldots , [x_{j}]\rangle $. 
 We obtain   
   \[v\langle [x_{1}], \ldots , [x_{t}]\rangle \bullet  y\langle [x_{t+1}], \ldots , [x_{j}]\rangle=\]
\[ (p-1)   v\langle [x_{1}], \ldots, [x_{t}]\rangle \odot  y\langle [x_{t+1}], \ldots , [x_{j}]\rangle+p\sum_{w, w'\in W}m_{w,w',v,y}'w\langle [x_{1}], \ldots , [x_{t}]\rangle \bullet w' \langle [x_{t+1}], \ldots , [x_{j}]\rangle ,\]
 for some integers $m_{w,w', v, y}'$.
  
   We can apply a similar argument to $a=w\langle [x_{1}], \ldots , [x_{t}]\rangle $ and 
   $c=  w' \langle [x_{t+1}], \ldots , [x_{j}]\rangle.$
   This can  then be substituted to the right hand side of the  above  equation (which has $u([x_{1}], \ldots [x_{j}])$ 
    on the left hand side) to obtain:
  \[u([x_{1}], \ldots  , [x_{j}])=(p-1)^{j -1}u\langle [x_{1}], \ldots  , [x_{j}]\rangle+p\sum_{w\in W}m _{w,u} w\langle [x_{1}], \ldots  , [x_{j}]\rangle .\]

  {\bf Part 2.} We are now ready to proof our result that 
   \[[a]\odot  [c]=(p-1)[a]\bullet  [c]+p\sum_{w\in V}\alpha _{w} w([a], [c]).\] We start with 
    where $w([a], [c])$ is the specialisation under the operation $\bullet $.

     Let $E_{[a],[c]}\subseteq A$ denote the set of products of some copies of element $[a]$ ($[a]$ appearing  at least once) and element $[c]$ at the end of each word ($[c]$ appearing only once) under the operation $\odot $.
      Moreover, we assume that both $[a]$ and $[c]$ appear in each product in the set $E_{[a],[c]}$. Let $V_{[a],[c]}$ be a vector whose entries are elements from $E_{[a],[c]}$ arranged in  such a way that longer products appear before shorter products.  
  Let $U_{[a],[c]}$ be the corresponding vector obtained from the corresponding products of $[a]$ and $[c]$ under the operation $\bullet $. 
  So for example if $([a]\odot [a])\odot ([a]\odot [c])$ is the $i$-th entry of $V_{[a],[c]}$ then $([a]\bullet [a])\bullet ([a]\bullet  [c])$ is the $i$-th entry of $U_{[a],[c]}$. 
  
  Let $RFM$ denote the set of row-finite matrices (where the rows and columns are  enumerated by the set of natural numbers) with integer entries. It is known that $RMF$ is a ring, and the product of any copies of matrices $M$ and $D$  is well defined and  belongs to $RFM$.

 By  Part $1$ above we obtain 
 
 \[U_{[a],[b]}=   DV_{[a],[b]}+ pMV_{[a],[b]}\] for some matrix $M$ in RFM and some diagonal matrix $D$ in RFM  whose diagonal entries are $(p-1)^{i-1}$ for some $i$. Notice that $(p-1)(-(1+p+p^{2}+\cdots +p^{p-1}))\equiv 1 \mod p^{p}$. Recall also that $p^{p-1}A=0$. 
 Therefore
 \[V_{[a],[b]}=D'U_{[a],[b]}-pD'MV_{[a],[b]},\]
   where $D'$ is a diagonal matrix  whose entries are $(-(1+p+p^{2}+\cdots +p^{p-1}))^{i}$.

  We can now substitute the above formula for $V_{[a], [b]}$ in the right hand side and obtain:
  \[V_{[a],[b]}=D'U_{[a],[b]}-pD'MD'U_{[a], [b]}+p^{2}D'MD'MV_{[a],[b]},\]
  for some matrix $M'$. 
   We can continue to substitute the expression for $V_{[a],[b]}$ on the right hand side. This process will stop after at most $p$ steps, since $p^{p-1}A=0$. This will give 
 \[V_{[a],[b]}=D^{-1}U_{[a],[b]}+pM'U_{[a],[b]},\]
  for some matrix $M'$ from RFM with integer entries.
   This concludes the proof. 
  \end{proof}

\section{ Relations between $*$ and $\odot $}

 In this section we will investigate some properties of the binary operation $\odot $ defined in Lemma \ref{6}. 
 Notice that if $p$ is a prime number and $1\leq i<p$ is a natural number then
 \[{p\choose i}={\frac pi}({{{p-1} \choose {i-1}}}),\] hence ${{{p-1} \choose {i-1}}/i}$ is an integer.

\begin{lemma}\label{7}
  Let $p$ be a prime number. Let $(A, +, \circ )$ be  a brace of cardinality $p^{n}$ with $p>n+1$. Let $a\in A$, define 
  \[f(a)=\sum_{i=1}^{p-1}({{{p-1} \choose {i-1}}/i})e_{i},\]
  where $e_{1}=a$, $e_{2}=a*a$ and $e_{i+1}=a*e_{i}$ for all $i$.
 Then $pf(a)=a^{\circ p}$. 
  Moreover, there are integers $\alpha _{1}, \ldots , \alpha _{p-1}$ which only depend on the additive group $(A, +)$ of the brace $A$ (and  do not depend on element $a$), and  such that 
\[   [a]=\sum_{i=1}^{p-1} \alpha _{i}[f_{i}(a)] ,\]
 where $[f_{1}(a)]=[f(a)]$, $[f_{2}(a)]=[f(a)]\odot [f(a)]$, $[f_{i+1}(a)]=[f(a)]\odot [f_{i}(a)]$, where $\odot $ is defined as in Lemma \ref{6}. 
   Moreover, $\alpha _{1}=1$.  As usual, by $[a]$ we mean $[a]_{ann(p^{2})}$. 
\end{lemma} 
\begin{proof} We will use a formula from Lemma \ref{citenote}, namely
 \[a^{\circ p}=\sum_{i=1}^{p-1}{p\choose k}e_{i},\] where  $e_{1}=a$ and $e_{i+1}=a*e_{i}.$ It works since $A^{p}=0$ as $n+1<p$. 
  Observe also that \[[f(a)]\odot [b]=[\wp^{-1}(a^{\circ p}*b)]\] by the definition of $\odot $.
 By using this formula we see that  $[f(a)]=[a]+\sum_{i>1}\beta _{i}[e_{i}]$ for some integers $\beta _{i}$. 
Similarly \[[f(a)]\odot [f(a)]=[\wp^{-1}(a^{\circ p}*f(a))]=[a*a]+\sum_{i>2}\beta _{i}'[e_{i}],\] for some integers $\beta _{i}'$. 
 We proceed by induction on $j$.  Assume that we have proved that 
  \[[f_{j}(a)]=[e_{j}]+\sum_{i>j}\beta _{i}''[e_{i}]\] for some integers $\beta _{i}''$.
  It follows that
  \[f_{j}(a)=e_{j}+\sum_{i>j}\beta _{i}''e_{i}+a',\]
   for some $a'\in ann(p^{2})$. Recall that for $x,y\in pA$ we have $[\wp^{-1}(x+y)]=[\wp^{-1}(x)]+\wp^{-1}(y)]$.   
 It follows that 
 \[[f_{j+1}(a)]=[f(a)]\odot [f_{j}(a)]=[\wp^{-1}(a^{\circ p}*(e_{j}+\sum_{i>j}\beta _{i}''e_{i}+a' ))]=\]
 \[[\wp^{-1}(a^{\circ p}*(e_{j}))] + \sum_{i>j}\beta _{i}''[\wp^{-1}(a^{\circ p}*e_{i})]+[\wp^{-1}(a^{\circ p}*a')]=\]
 \[=[e_{j+1}]+\sum_{i>j+1}\beta _{i}'''[e_{i}]\] for some integers $\beta _{i}'''$.
  Notice that  $[\wp^{-1}(a^{\circ p}*a')]=[0]$ since $p^{2}\wp^{-1}(a^{\circ p}*a')=p(a^{\circ p}*a')=0$, since $a'\in ann(p^{2})$.

 Let $f$ be the vector whose entries are elements $[f_{i}(a)]$ and let $E$ be the vector whose entries are elements $[e_{i}]$.
 Then
 \[f=ME\] for some upper triangular matrix (with integer entries), whose diagonal entries are $1$.
 It follows that $E=M'f$ for some upper diagonal matrix with integer entries with $1$'s on the diagonal 
  (because $p^{n}A=0$). By looking at the first entry of $E$ and the first entry of $M'f$ we get the required conclusion. 
\end{proof} 
\begin{proposition}
 Let notation and assumptions be as in Lemma \ref{7}, and let $b\in A$. Then
  there are integers $\gamma _{1}, \ldots , \gamma _{p-1}$ which only depend on the additive group $(A, +)$ of the brace $A$ and do not depend on element $a$ such that 
\[   [a*b]=\sum_{i=1}^{p-1} \gamma _{i}[q_{i}(a,b)] \]
 where $[q_{1}(a)]=[f(a)]\odot [b]$, $[q_{2}(a)]=[f(a)]\odot  [q_{1}(a,b)]$, 
 $[q_{i+1}(a,b)]=[f(a)]\odot [q_{i}(a, b)]$, where $\odot $ is defined as in Lemma \ref{6}. 
   Moreover, $\gamma _{1}=1$.
    As usual by $[a]$ we mean $[a]_{ann(p^{2})}$. 
\end{proposition}
\begin{proof}  The proof is similar to the proof of Lemma \ref{7}. By a formula from Lemma \ref{citenote}, 
\[a^{\circ p}*b=\sum_{i=1}^{p-1}{p\choose k}e_{i}'\]  where  $e_{1}'=a*b$ and $e_{i+1}'=a*e_{i}'$ for $i=1,2, \ldots , p-2$. This formula works since $n+1<p$.
 Notice that 
\[[q_{1}(a,b)]=[f(a)]\odot [b]=[\wp^{-1}((pf(a))*b)]=[\wp^{-1}(a^{\circ p}*b)]=   [a*b]+\sum_{i>1}\beta _{i}[e_{i}']\] for some integers $\beta _{i}$. We will proceed by induction on $j$. 
 Assume that 
 \[[q_{j}(a,b)]=[e_{j}']+\sum_{i>j}\beta _{i}'[e_{i}']\] for some integers $\beta _{i}'$.
 Reasoning similarly as in the proof of Lemma \ref{7} we can show that 
 \[[q_{j+1}(a,b)]=[f(a)]\odot [q_{j}(a,b)]=[e_{j+1}']+\sum_{i>j+1}\beta _{i}''[e_{i}']\] for some integers $\beta _{i}''$.
 Let $f$ be the vector whose entries are elements $[q_{i}(a,b)]$ and let $E$ be the vector whose entries are elements $[e_{i}']$. 
 Then
 \[f=ME\] for some upper triangular matrix (with integer entries) whose diagonal entries are $1$.
 It follows that \[E=M'f\] for some upper diagonal matrix with integer entries with $1$'s on the diagonal 
  (because $p^{n}A=0$). By looking at the first entry of $E$ and the first entry of $M'f$ we get the required conclusion.
 This concludes the proof.
\end{proof}

\section{ Some properties of function $f$}
  Let $p$ be a prime number. Let $(A, +, \circ )$ be  a brace of cardinality $p^{n}$ with $p>n+1$. Let $a\in A$. Recall that 
  \[f(a)=\sum_{i=1}^{p-1}({{{p-1} \choose {i-1}}/i})e_{i},\]
  where $e_{1}=a$, $e_{2}=a*a$ and $e_{i+1}=a*e_{i}$ for all $i$.
 Then $pf(a)=a^{\circ p}$.
 
 In this section we investigate properties of this function.

\begin{theorem}\label{f(a)}
  Let $p$ be a prime number. Let $(A, +, \circ )$ be  a brace of cardinality $p^{n}$ with $p>n+1$. For $a\in A$ let  
  $f(a)=\sum_{i=1}^{p-1}({{{p-1} \choose {i-1}}/i})e_{i}.$
 Then the map 
 \[[a]\rightarrow [f(a)]\] 
 is an injective function on $A/ann(p^{2})$.
  Consequently, since the set $A/ann(p^{2})$ is finite, it follows that this function is a bijection.
 As usual we denote $[a]=[a]_{ann(p^{2})}$.
 \end{theorem}  
 \begin{proof} Let $a,b\in A$. Suppose that $[f(a)]=[f(b)]$ then $f(a)-f(b)\in ann(p^{2})$, hence 
 $p(p(f(a)-pf(b))=0$, so $pf(a)-pf(b)\in ann(p)$. Recall that $pf(a)=a^{\circ p}$, hence 
 \[a^{\circ p}=b^{\circ p}+e',\] for some $e'\in ann(p)$.
  
   We will now show that all products 
 \[[x_{1}*(x_{2}(*\cdots *(x_{k-1}*x_{k})\cdots ))]\] for $x_{1}, \ldots , x_{k}\in\{a,b\}$
 are equal. 
   
 For $k=n+1$ the result is true because all such products of length $n+1$ will be zero since
  $A^{n+1}=0$. We proceed by induction  on $k$ in the decreasing order. Let $i$ be a natural number, $i<n+1$.
   We will show the result is true for $k=i$ provided that it is true for all  
 $k>i$. 
 
  We will first show that 
 \[[a*(x_{1}*(x_{2}(*\cdots *(x_{k-2}*x_{k-1})\cdots )))]=[b*(x_{1}*(x_{2}(*\cdots *(x_{k-2}*x_{k-1})\cdots )))]\] for 
  all  $x_{1}, \ldots x_{k}\in \{a,b\}$.

 Observe that   $a^{\circ p}-b^{\circ p}\in ann(p)$ yields  
 \[a^{\circ p}*(x_{1}*(x_{2}(*\cdots *(x_{k-2}*x_{k-1})\cdots ))=\]
 \[b^{\circ p}*(x_{1}*(x_{2}(*\cdots *(x_{k-2}*x_{k-1})\cdots ))+e,\]
  for some $e\in ann(p)$, since $ann(p)$ is an ideal in the brace $A$ by \cite{passage}. It follows 
  from Lemma \ref{co} applied to $a'=a^{\circ p}$ and $b'=b^{\circ p}-a^{\circ p}$ and $c'=x_{1}*(x_{2}(*\cdots *(x_{k-2}*x_{k-1})\cdots )$. 
 
  By Lemma \ref{citenote} we have:
 \[a^{\circ p}*(x_{1}*(x_{2}(*\cdots *(x_{k-2}*x_{k-1})\cdots ))=\]
 \[=pa*(x_{1}*(x_{2}(*\cdots *(x_{k-2}*x_{k-1})\cdots )))+{\frac {p(p-1)}2}a*(a*(x_{1}*(x_{2}(*\cdots *(x_{k-2}*x_{k-1})\cdots )))+\cdots .\]
 
 Similarly,
 \[b^{\circ p}*(x_{1}*(x_{2}(*\cdots *(x_{k-2}*x_{k-1})\cdots ))=\]
 \[=pb*(x_{1}*(x_{2}(*\cdots *(x_{k-2}*x_{k-1})\cdots ))) +{\frac {p(p-1)}2}b*(b*(x_{1}*(x_{2}(*\cdots *(x_{k-2}*x_{k-1})\cdots ))) +\cdots .\]  
 
  The above three equations combined together  imply after  applying $\wp^{-1}$ to both sides  
  \[[a*(x_{1}*(x_{2}*(\cdots *(x_{k-2}*x_{k-1})\cdots )))+{\frac {p-1}2}a*(a*(x_{1}*\cdots *(x_{k-2}*x_{k-1})\cdots ))+\cdots ]=\]
  \[[b*(x_{1}*(x_{2}*(\cdots *(x_{k-2}*x_{k-1})\cdots ))) +{\frac {p-1}2}b*(b*(x_{1}*\cdots *(x_{k-2}*x_{k-1})\cdots ))+\cdots ] +[\wp^{-1}(e)].\]
 
 Notice that $[\wp^{-1}(e)]=[0]$ since $p^{2}\wp^{-1}(e)=pe=0$. 
 Notice that by the inductive assumption
 \[[a*(a*(x_{1}*(*\cdots *(x_{k-2}*x_{k-1})\cdots )))]=[b*(b*(x_{1}*(*\cdots *(x_{k-2}*x_{k-1})\cdots )))].\]
  This also holds for the next products in the above sum (involving more $a$ and $b$) by the inductive assumption.
  
  The two above arguments show that 
 \[[a*(x_{1}*(x_{2}(*\cdots *(x_{k-2}*x_{k-1})\cdots )))]=[b*(x_{1}*(x_{2}(*\cdots *(x_{k-2}*x_{k-1})\cdots )))]\] for 
  this fixed $k$ and for all $x_{1}, \ldots x_{k-1}\in \{a,b\}$.  Notice now that we can use a similar argument by putting in the ith place  $a^{\circ p}$ on the left-hand side and 
 $b^{\circ p}$ on the right hand-side, without changing the elements $x_{1}, x_{2}, \ldots, x_{k-1} $:
 \[x_{1}*(\cdots *(x_{j}*(a^{\circ p}*(x_{j+1}*(\cdots *(x_{k-2}*x_{k-1})\cdots ))))\cdots )=\]
 \[x_{1}(*\cdots *(x_{j}*(b^{\circ p}*(x_{j+1}*(*\cdots *(x_{k-2}*x_{k-1})\cdots ))))\cdots )+e''\] for 
  all $x_{1}, \ldots x_{k-1}\in \{a,b\}$, for some $e''\in A$. Reasoning similarly as when we did to show  that $e\in ann(p)$, we obtain that $e''\in ann(p)$. 
  By Lemma \ref{citenote} we have 
  \[a^{\circ p}*(x_{j+1}*(\cdots *(x_{k-2}*x_{k-1})\cdots ))=\sum_{i=1}^{p-1}{p\choose i}y_{i}\] where
  $y_{1}= a*(x_{j+1}*(\cdots *(x_{k-2}*x_{k-1})\cdots )))$, and $y_{i+1}=a*y_{i}$ for $i=1,2, \ldots , p-2$.
   Similarly 
  \[b^{\circ p}*(x_{j+1}*(\cdots *(x_{k-2}*x_{k-1})\cdots )))=\sum_{i=1}^{p-1}{p\choose i}y_{i}'\] where
  $y_{1}'= b*(x_{j+1}*(\cdots *(x_{k-2}*x_{k-1})\cdots )))$, and $y_{i+1}=b*y_{i}$ for $i=1,2, \ldots , p-2$.
  
  By applying  $\wp ^{-1}$ to both sides of equation 
  \[x_{1}*(\cdots *(x_{j}*(a^{\circ p}*(x_{j+1}*(\cdots *(x_{k-2}*x_{k-1})\cdots ))))\cdots )=\]
 \[x_{1}(*\cdots *(x_{j}*(b^{\circ p}*(x_{j+1}*(*\cdots *(x_{k-2}*x_{k-1})\cdots ))))\cdots )+e''\] 
  we obtain after taking cosets 
 \[[x_{1}*(\cdots *(x_{j}*\sum_{i=1}^{p-1}\alpha _{i}y_{i})\cdots )]=[x_{1}*(\cdots *(x_{j-1}*\sum_{i=1}^{p-1}\alpha _{i}y_{i}')\cdots )],\] where ${p\choose i}=p\alpha _{i}$ for each $i$, so $\alpha _{1}=1$.  
    By the inductive assumption, 
 \[[x_{1}*(\cdots *(x_{j}*y_{i})\cdots )]=[x_{1}*(\cdots *(x_{j}*y_{i}')\cdots )]\] for $i>1$.
Therefore
 \[[x_{1}*(\cdots *(x_{j}*y_{1})\cdots )]=[x_{1}*(\cdots *(x_{j-1}*y_{1}')\cdots )].\] 
 By the definition of $y_{1}$, we obtain  
 \[[x_{1}*(\cdots *(x_{j}*(a*(x_{j+1}*(\cdots *(x_{k-2}*x_{k-1})\cdots ))))]=\]
 \[[x_{1}*(\cdots *(x_{j}*(b*(x_{j+1}*(*\cdots *(x_{k-2}*x_{k-1})\cdots ))))]\] for 
  all $x_{1}, \ldots x_{k-1}\in \{a,b\}.$

  We have thus proved that 
 \[[x_{1}*(x_{2}(*\cdots *(x_{k-1}*x_{k})\cdots ))]=[x_{1}'*(x_{2}'(*\cdots *(x_{k-1}'*x_{k}')\cdots ))],\] 
  provided that $x_{j}=x_{j}'$ in all places except in one place, and all $x_{i}, x_{i}'\in \{a,b\}$. 

 Observe now that when we change one index at a time this implies:
 \[[x_{1}*(x_{2}(*\cdots *(x_{k-1}*x_{k})\cdots ))]=[x_{1}*(x_{2}(*\cdots *(x_{k-1}*x_{k}')\cdots ))],\]
 and then  
 \[[x_{1}*(x_{2}(*\cdots *(x_{k-1}*x_{k}')\cdots ))]=[x_{1}*(x_{2}(*\cdots *(x_{k-2}*(x_{k-1}'*x_{k}'))\cdots )))],\]
 and  continuing to change the $i$-th  digit at $i$-th  time we eventually get 
 \[[x_{1}'*(x_{2}'(*\cdots *(x_{k-1}'*x_{k})\cdots ))]=[x_{1}'*(x_{2}'(*\cdots *(x_{k-1}'*x_{k}')\cdots ))].\] 
  These equations imply that 
 
 \[[x_{1}*(x_{2}(*\cdots *(x_{k-1}*x_{k})\cdots ))]=[x_{1}'*(x_{2}'(*\cdots *(x_{k-1}'*x_{k}')\cdots ))].\]
 This proves the inductive argument.

 Notice that for $k=1$ we have $[a]=[b]$, as required.
 \end{proof}
 \begin{theorem}\label{g(a)}
  Let $p$ be a prime number. Let $(A, +, \circ )$ be  a brace of cardinality $p^{n}$ with $p>n+1$. For $a\in A$
 define \[g(a)=f^{(p^{p}!-1)}(a)\] where $f^{(1)}(a)=f(a)$ and for every $i$ we denote $f^{(i+1)}(a)=f(f^{(i)}(a))$. 
 Then \[[f(g(a))]=[g(f(a))]=[a].\]
  Moreover
  \[[f(x)]=[g^{(p^{p}!-1)}(a)]\]  where $g^{(1)}(a)=a$ and for every $i$ we denote $g^{(i+1)}(a)=g(g^{(i)}(a))$.
 \end{theorem}
 \begin{proof} 
  Observe first that since $[a]\rightarrow [f(a)]$ is a bijection on $A/ann(p^{2})$, then 
  for every $i$, $[f^{(i)}(a)]\rightarrow [f^{(i+1)}(a)]$ is a bijective function on $A/ann(p^{2})$, and hence for every $k>0$
  \[[a]\rightarrow [f^{(k)}(a)]\] is a bijective function on $A/ann(p^{2})$.  
 
Notice that there are $p^{n}+1$ elements  $[f^{(1)}(a)], [f^{(2)}(a)], \cdots , [f^{(p^{n}+1)}(a)]$, and since brace $A$ has cardinality $p^{n}$ it follows that \[[f^{(i)}(a)]=[f^{(j)}(a)],\]
for some $1\leq i<j\leq p^{n}+1$, hence $j-i\leq p^{n}\leq p^{p}$.
 Notice that the function
  \[[f^{(i)}(a)]\rightarrow [a] \] is a bijective function on $A/ann(p^{2})$. 
  Applying this function to both sides of equation
  \[[f^{(i)}(a)]=[f^{(j)}(a)],\]
  we obtain 
 \[[a]=[f^{(j-i)}(a)].\]
 Notice that $j-i\leq p^{n}$ implies that $j-i$ divides $p^{n}!$, so it divides $p^{p}!$. Therefore 
  \[[f^{(p^{p}!)}(a)]=[a].\]
   This shows that \[[f(g(a))]=[g(f(a))]=[a].\]
 
 On the other hand, by the first part of this proof it follows that  $[a]\rightarrow [g(a)]$ is a bijective function, since   \[[g(a)]=[f^{(p^{p}!-1)}(a)].\]
 
 Observe now that \[[g^{(p^{p}!-1)}(a)]=[f^{((p^{p}!-1)^{2})}(a)]=[f(a)],\]
  since $[f^{(p^{p}!)}(a)]=[a]$.
 \end{proof}

\section{ Recovering braces from pre-Lie algebras}

We are now able to prove the following theorem:

\begin{theorem}\label{789}  Let $p$ be a prime number and $n<p-1$ be a natural number. Let $(A, +)$ be an abelian group of cardinality $p^{n}$.  Let $V$ denote the set of all non-associative words in non-commuting variables $X,Y$ (where $Y$ appears only once at the end of each word and $X$ appears at least once in each word in $V$). Then there are integers $e _{w}$, for $w\in V$, such that only a finite number of them is non-zero and that the following holds: 
 For each brace $(A, +, \circ )$  with the additive group $(A, +)$ and for each $a, c\in A$ 
 we have 
\[[a]* [c]=\sum_{w\in V}e _{w}w([a],[c]),\]  where  $w([a],[c])$ is a specialisation of the word $w$ for $X=[a]$, $Y=[c]$, and the multiplication in $w(a,c)$ is the same as the  multiplication in the pre-Lie ring $(A/ann (p^{2}), +, \bullet )$ constructed in Theorem \ref{1} from the brace $(A, +, \circ )$. So for example if $w=((XX)X)Y$ then $w(a,c)=(([a]\bullet [a])\bullet [a])\bullet [c])$.  
\end{theorem}
\begin{proof}   
  By Lemma \ref{7} we have that \[[a]=\sum_{i=1}^{p-1} \alpha _{i}[f_{i}(a)]\] where $[f_{i+1}(a)]=[f(a)]\odot [f_{i}(a)]$. 
   By applying it for $a=g(a)$, where $g(a)$ is as in Theorem \ref{g(a)}, so $g(a)=f^{(p^{p}!-1)}(a)$ and $[f(g(a))=[a]$, we obtain:
   \[[g(a)]=\sum_{i=1}^{p-1} \alpha _{i}[h_{i}(a)]\] where $[h_{1}(a)]=[a]$ and $[h_{i+1}(a)]=[a]\odot [h_{i}(a)]$ and $\alpha _{i}$ are some integers. Observe that integers  $\alpha _{i}$ only depend on the additive group $(A, +)$ of the brace $A$
    and do not depend on the multiplicative group $(A, \circ )$ and do not depend on element $a$. Therefore $[g(a)]$ can be obtained by applying the operations $+$ and $\odot $ several times to some copies of element $[a]$, and the method is the same for all braces with the same additive group $(A, +)$.
    
    {\bf Fact 1.}  By Theorem \ref{g(a)} we obtain 
    $[f(a)]=[g^{p^{p}!-1}(a)]$, therefore $[f(a)]$ can be obtained by applying the operations $\odot $ and $+$ several times to some copies of element $[a]$, and the method and order of applying 
     these operations does not depend on element $a$, and the same method works for all braces with the additive group $(A, +)$. 

    {\bf Fact 2.}  By Lemma \ref{7} we have that $[a*b]=[a]*[b]$ can be written by using operations $\odot $ and $+$ applied several times  to  some copies of element $[f(a)]$ and to the element  $[b]$ and to sums of the obtained elements 
    (and the method does not depend on elements $f(a)$ and $b$) and the same method works for all braces with the additive group $(A, +)$. 
 
  $ $
      
  By combining the above Fact $1$ and Fact $2$ we obtain that $[a]*[b]$ can be obtained by applying operation $\odot $ to elements $a$ and $b$, and the algorithm for which order to apply these operations $\odot, +$  does not depend on elements $a, b$. Moreover  the same method works for all braces with the additive group $(A, +)$.

  By using Theorem \ref{1}  we can write product $\odot $ by using the pre-Lie operation $\bullet $ and $+$. Notice that the operation $+$ is this pre-Lie ring is the same as the addition in the factor brace $(A/ann(p^{2}), +, \circ )$, so it only depends on the multiplicative group of brace $A$.
  Since $\bullet $ is a pre-Lie algebra multiplication it is distributive with respect to the addition, the obtained result 
  can be simplified and written as 

  \[\sum_{w(a,b)\in W}\beta _{w}w(a,b),\] where $W$ is the set of products of some number of element $[a]$ and element $[b]$ at the end under the pre-Lie operation $\bullet $, 
 where $\beta _{w}$ are some integers which don't depend on elements $a, b$ and do not depend on  the multiplicative group of  the brace $A$ (and only depend on the additive group $(A, +)$). This concludes the proof. 
 \end{proof}

\section{ Braces $A/ann(p^{4})$ and groups of flows }

In this section we will prove Theorem \ref{Main}.
 Notice that, as shown in \cite{passage},  the construction of the group of flows is well defined for every left nilpotent pre-Lie ring $(A, +, \cdot)$. Moreover, the group of flows  is a brace with the same addition as in the original pre-Lie ring, and we will call this brace just the group of flows of the pre-Lie ring $(A, +, \cdot)$.
 The construction of the group of flows was first introducted in \cite{AG}. The connection between braces and groups of flows was discovered by Wolfgang Rump in 2014 \cite{Rump}. 
 
We begin with a Lemma similar to the result obtained in the last section of  \cite{shsm}.
 \begin{lemma}\label{similar} 
 Let $(A, +, \cdot)$ be a left nilpotent pre-Lie ring of cardinality $p^{n}$ for some prime number $p$ and some natural number $n<p-1$. Let $(A/ann(p^{2}), +, \cdot)$ be the factor pre-Lie ring by the ideal $ann(p^{2})=\{a\in A:p^{2}a=0\}$. Elements of this pre-Lie ring are cosets $[a]=\{a+i:i\in ann(p^{2})\}$. 
 We denote the multiplication and addition in this pre-Lie ring with the same symbols a in the original ring:
 \[[a]\cdot [b]=[a\cdot b], [a]+[b]=[a+b].\]
 Let $(A, +, \circ)$ be the brace obtained from the pre-Lie ring $(A, +, \cdot)$ using the construction of the group of flows. 
 Let $(A/ann (p^{2}), +, \bullet )$ be the corresponding pre-Lie ring, constructed as in \cite{shsm} from the brace $(A, +, \circ )$ (notice that this construction is quoted in Theorem \ref{1}). 
   Then $(A/ann (p^{2}), +, \bullet )$ and $(A/ann(p^{2}, +, \cdot)$ are the same pre-Lie rings, so
   \[[x]\bullet [y]=(p-1)([x\cdot y]),\] for $[x], [y]\in A/ann(p^{2})$. 
 \end{lemma}
 \begin{proof} Notice that the addition in both pre-Lie rings is the same as the addition in the additive group $A/ann(p^{2})$ of the factor brace $A/ann(p^{2})$. We need to show that \[[x]\bullet [y]=(p-1)([x\cdot y]),\] for $[x], [y]\in A/ann(p^{2})$. We will use the same notations for the construction of the group of flows as in \cite{passage}. Recall that the formula for the operation $*$ in the group of flows of pre-Lie algebra $(A, +, \cdot )$ is 
 \[a*b= {\Omega (a)}\cdot b+{\frac 1{2!}}{\Omega (a)}\cdot ({\Omega (a)}\cdot b)+ {\frac 1{3 !}}{\Omega (a)}\cdot ({\Omega (a)}\cdot ({\Omega (a)}\cdot b))+\ldots  \]
 where $a\circ b=a*b+a+b$ in brace $(A, +, \circ)$. 
  By Lemma \ref{aga} ($11$ from \cite{passage}) we have that 
$\Omega (a)=a+\sum_{w }\alpha _{w}w(a)$ for some integers $\alpha _{w}$, where $w$ are finite  non associative words in variable $x$ (of degree at least $2$), and $w(a)$ is a specialisation of $w$ at $a$ (so for example if $w=x\cdot (x\cdot x)$ then $w(a)=a\cdot (a\cdot a)$). Moreover, there is $m$ such that $\alpha _{j}=0$ provided that  $j>m$. 

 Let $a,b\in A$, we will denote $a'=pa$. Observe now that  
\[[\xi ^{i}a]\odot [b]=[\wp^{-1}((\xi ^{i}a')*b]=[\wp^{-1}({\Omega (\xi ^{i}a')}\cdot b+{\frac 1{2!}}{\Omega (\xi ^{i}a')}\cdot ({\Omega (\xi ^{i}a')}\cdot b)+\ldots )].\]
 Observe also that 
\[[\Omega (\xi ^{i}a')]=\xi ^{i}a'+\sum_{w }\alpha _{w}w(\xi ^{i}a')=\xi ^{i}a+\sum_{k=2}^{m}(\xi ^{i})^{k}f_{k}(a'),\]
 where \[f_{k}(a')=\sum_{w\in W_{k}}\alpha _{w}w(a')\] where $W_{k}$ consists of words of length $k$.

 Therefore, \[[\xi ^{i}a]\odot [b]=[\wp^{-1}( \xi ^{i}a'+\sum_{k=2}^{mp}(\xi ^{i})^{k}t_{k}(a',b))],\]
where $t_{k}(a',b)$ is a linear combination (with integer coefficients) of some products of $j$ copies of $a'$ and element $b$ at the end under the operation $\cdot $.
 Notice, that  for $j>p-1$ each such product will be zero because it will belong to $p^{p-1}A=0$, because  $a'=pa$ and the operation $\cdot $ is distributive. Therefore,
\[[\xi ^{i}a]\odot [b]=[\wp^{-1}( \xi ^{i}a'\cdot b+\sum_{k=2}^{p-2}(\xi ^{i})^{k}t_{k}(a',b))],\]

  Notice that 
\[\xi ^{p-1-i}[\xi ^{i}a]\odot [b]=[\wp^{-1}( a'\cdot b+\sum_{k=2}^{p-2}(\xi ^{k-1})^{i}t_{k}(a',b))].\]
 
Recall also the formula $[a]\bullet [b]=\sum_{i=0}^{p-2}\xi ^{p-1-i}[\xi ^{i}a]\odot [b]$, where  $\xi ^{p-1}\equiv 1 \mod p^{n}$.
 Combining the above equations together we get 

\[[a]\bullet [b]=[\wp^{-1}( \sum_{i=0}^{p-2} (a'\cdot b+\sum_{k=2}^{p-2}(\xi ^{k-1})^{i}t_{k}(a',b)))].\]
Notice that $(1-\xi ^{i})\sum_{j=1}^{p-2}(\xi ^{i})^{j}=(\xi ^{i})^{p-1}-1 =0$, provided that $0<i<p-1$. 
Consequently \[[a]\bullet [b]=[\wp^{-1}(\sum_{i=0}^{p-1}a'\cdot b)]=[\wp^{-1}((p-1)a'\cdot b    )]=(p-1) [a\cdot b],\]
 since $\cdot $ is a pre-Lie algebra operation, so 
\[a'\cdot b=(pa)\cdot b=p(a\cdot b).\] This concludes the proof.

 \end{proof}
 
 {\bf Remark 2.} Let $(A, +, \circ _{1})$ and $(A, +, \circ _{2})$ be two braces with the same additive group $(A, +)$.
  Suppose that the cardinality of $A$ is $p^{n}$ for some natural number $n$, and some prime number $p>n+1$. 
  Let $(A/ann(p^{2}), +, \bullet _{1})$ and $(A/ann(p^{2}), +, \bullet _{2})$ be the corresponding pre-Lie rings constructed as in \cite{shsm} (see  Theorem \ref{1} for details).
  Note that the set $ann(p^{2})=\{a\in A: p^{2}a=0\}$ is the same in both braces. Therefore the set $A/ann(p^{2})$ can be defined only by using the additive group $(A, +)$ without using operations $\circ _{1}, \circ _{2}$.

 Let $(A/ann (p^{2}), +, \circ _{1})$ be the factor brace obtained from the brace 
  $(A, +, \circ _{1})$ and  let $(A/ann (p^{2}), +, \circ _{2})$ be the factor brace obtained from brace 
  $(A, +, \circ _{2})$. As usual we use the same notation for operations $+, \circ $ in the brace and in the factor brace.

 Observe that by Theorem \ref{789} if $x \bullet_{1} y=x \bullet _{2} y$ for all $x,y\in A/ann(p^{2})$ then 
 $x\circ _{1} y=x\circ _{2}y$, for $x,y\in A/ann(p^{2})$, and hence the braces 
 $(A, +, \circ _{1})$ and $(A/ann (p^{2}), +, \circ _{2})$ are the same. 

$ $

{\bf Proof of Theorem \ref{Main}.}   As usual we  will denote the operations of the addition and of the multiplication  in the factor brace by the same symbols as in the original brace.  Let \[(A, +, \circ )\]  be 
 a brace of cardinality $p^{n}$ for some natural number $n$, and some prime number $p>n+1$.
   Recall that $ann(p^{2})=\{a\in A:p^{2}a=0\}$. 
  Let \[(A/ann (p^{2}), +, \bullet )\] be the pre-Lie ring, constructed as in \cite{shsm} from the brace $(A, +, \circ )$  
  (notice that this construction is quoted in Theorem \ref{1}). 
   We will denote the subset of  elements of order $p^{2}$, in the additive group $(A/ann (p^{2}), +)$, by $I$. 
   Observe that   \[I=\{[a]\in A/ann(p^{2}): [p^{2}a]=[0].\}\]
    Observe that we can define the factor pre-Lie ring $((A/ann (p^{2}))/I, +, \bullet _{1} )$ of the pre-Lie ring $(A/ann (p^{2}), +, \bullet )$ by the ideal $I$. In this pre-Lie ring 
  $((A/ann (p^{2}))/I, +, \bullet _{1} )$ we have  $[x]_{I}\bullet _{1} [y]_{I}=[x\bullet y]_{I}$, $[x]_{I}+[y]_{I}=[x+y]_{I}$ 
  for $x,y\in A/ann(p^{2})$
   (where $[x]_{I}, [y]_{I}$ are elements of this pre-Lie ring $(A/ann (p^{2}))/I$).  
   We will call the pre-Lie ring \[((A/ann (p^{2}))/I, +, \bullet _{1} )\]  the  {\em pre-Lie ring $1$}.

  On the other hand, let \[(A/ann(p^{2}), +, \circ )\]  be the factor brace of the brace $(A, +, \circ )$, by the ideal 
  $ann(p^{2})$. 
  Recall that $I$ denotes the set of elements of the additive order at most $p^{2}$ in the group
  $(A/ann(p^{2}), +)$. 
   Let 
  \[((A/ann(p^{2}))/I, + ,\bullet _{2})\] be the pre-Lie ring constructed  as in \cite{shsm} (so constructed as in  Theorem \ref{1}) from  the brace 
  $(A/ann(p^{2}), +, \circ )$. We will call this pre-Lie ring the {\em pre-Lie ring $2$}. 
 We also consider the factor brace of the brace \[(A/ann(p^{2}), +, \circ )\] by the ideal $I$. We will call this  {\em brace $2$}, and denote it as $((A/ann(p^{2}))/I, +, \circ )$.

  $ $
  
   We will now show that pre-Lie ring $1$ is the same as  pre-Lie ring $2$. Observe that the additive groups of these pre-Lie rings are the same since the construction of the group of flows does not change the additive group.
  Observe that for $x,y\in A$ by $[x], [y] $ we denote elements of the additive group $(A/ann (p^{2}), +)$ and 
   by $[[x]]_{I}, [[y]]_{I}$ we denote elements of the additive group $((A/ann (p^{2}))/I, +)$.
   We need to show that 
   \[[[x]]_{I}\bullet _{1} [y]_{I}=[x]]_{I}\bullet _{2} [[y]]_{I},\] for $x,y\in A/ann(p^{2})$.
    
    Observe that 
   \[[x]_{I}\bullet _{2} [y]_{I}=[[x]\bullet [y] ]_{I}=[[\wp^{-1}((px)\cdot y)]]_{I},\]
   where  for $x,y\in A$ we define $x\cdot y=\sum_{i=0}^{p-2}\xi ^{p+1-i}((\xi ^{i}x)*y).$

    Observe that 
   \[[x]_{I}\bullet _{2} [y]_{I}=  [\wp^{-1}([px] \cdot  [y])]_{I},\]
    where for $x,y\in A$ we define $[x]\cdot [y]=\sum_{i=0}^{p-2}\xi ^{p+1-i}([\xi ^{i}x]*[y])$.
    
    By the definition of $I$, we know that if $z, w\in A/ann(p^{2})$ then $[z]_{I}=[w]_{I}$ if and only if $z-w\in I$,
    which means $p^{2}z=p^{2}w$.
      
      Therefore, to show that $[[x]]_{I}\bullet _{1} [y]_{I}=[[x]]_{I}\bullet _{2} [[y]]_{I}$ it suffices to show that 
    \[p^{2}[\wp^{-1}((px)\cdot y)]=p^{2}\wp^{-1}([px] \cdot  [y]).\]
    This is equivalent to
    \[p[(px)\cdot y]=p([px] \cdot  [y]),\] and this is true by the definition of the factor brace $(A/ann(p^{2}), +, *)$.
    Consequently, operations $\bullet _{1}$ and $\bullet _{2}$ are the same.
    It follows that the pre-Lie ring $1$ is the same as the pre-Lie ring $2$.
    
 We will now introduce  {\em pre-Lie ring $3$}. Consider the pre-Lie ring $(A/ann(p^{2}), +, \bullet _{3})$
  such that the addition in this pre-Lie ring is the same as the addition in the pre-Lie ring $1$, and the multiplication is defined as
 \[[x]\bullet _{3} [y]=-(1+p+p^{2}+\ldots p^{p-1})([x]\bullet _{1} [y]),\] 
 for $x,y\in A$. Notice that this gives  a well defined pre-Lie ring (see \cite{Lazard} for a proof of this remark).
 We will call the pre-Lie ring  $(A/ann(p^{2}), +, \bullet _{3})$  {\em pre-Lie ring $3$}. 

 We will now introduce {\em  brace $1$}. 
  Let $(A/ann (p^{2}), +, \circ _{3})$ be the brace which is constructed as the group of flows from    pre-Lie ring $3$. 
  Let $((A/ann (p^{2}))/I, +, \bullet _{4})$ be the pre-Lie ring constructed from the  brace $(A/ann (p^{2}), +, \circ _{3})$   by the construction from Theorem \ref{1} (so  by the construction from \cite{shsm}). We will call this pre-Lie ring \[((A/ann (p^{2}))/I, +, \bullet _{4})\] the {\em pre-Lie ring $4$}.
 By Lemma \ref{similar} we obtain that the addition in the pre-Lie ring
 $1$ is the same as the addition in the pre-Lie ring $4$. 
 Observe also that by Lemma \ref{similar} we have for $a,b\in A/ann (p^{2})$ 
\[[a]_{I}\bullet _{4} [b]_{I}=(p-1)([a]_{I}\bullet _{3}[b]_{I})=-(1+p+p^{2}+\cdots )(p-1)([a]_{I}\bullet _{1}[b]_{I})=[a]_{I}\bullet _{1} [b]_{I},\] 
since $p^{n}a=a$ for every $a\in A/ann(p^{2})$. Therefore the pre-Lie ring $4$ 
 is the same as the pre-Lie ring $1$.

Let {\em brace $1$} be the factor brace of the brace $(A/ann (p^{2}), +, \circ _{3})$ by the ideal $I$. Observe that brace $1$ is the group of flows of  pre-Lie ring $5$ which is the factor ring of  pre-Lie ring $3$ by the ideal $I$ (because the brace $(A/ann (p^{2}), +, \circ _{3})$
 is the group of flows of  pre-Lie ring $3$).

    Recall that  pre-Lie ring $1$ is the same as  pre-Lie ring $2$. Moreover, pre-Lie ring $4$ is the same as  pre-Lie ring $1$. Therefore  pre-Lie ring $4$ is the same as  pre-Lie ring $2$. 

 By using Theorem \ref{789} we obtain  that 
      brace $1$ and brace $2$ are the same, since it is possible to recover brace $1$ from pre-Lie ring $4$, and brace $2$ from pre-Lie ring $2$ by the formula from Theorem \ref{789}. 
     
     Therefore the brace $2$ is the group of flows, since brace $1$ is the group of flows  of a left nilpotent  pre-Lie algebra.

    It remains to show that  brace $2$ is the same as the factor brace of the brace $(A, +, \circ )$  by the ideal $ann(p^{4})$.  
    Notice that we can map $[[a]]_{I}\rightarrow [a]_{ann(p^{4})}$ for $a\in A$. Notice that this map is well defined because 
    if $[[a]]_{I}=[[b]]_{I}$ that means that $[a]-[b]$ is in the $I$, so $[p^{2}a]=[p^{2}b]$. This in turn means that $p^{2}(p^{2}a)=p^{2}(p^{2}b)$, which means $[a]_{ann (p^{4})}=[b]_{ann(p^{4})}$. 
    Notice that this map is a homomorphism of braces since 
    $[[a]]_{I}*[[b]]_{I}=[[a*b]]_{I}\rightarrow [a*b]_{ann(p^{4})}=[a]_{ann(p^{4})}*[b]_{ann(p^{4})}$. 
    This shows that  brace $2$ is the same as  the factor brace  $(A/ann(p^{4}), +, \circ )$.

\section{ Some results from other papers which were used in previous sections}

For the convenience of the reader we recall  some results  from other papers used in previous sections.
All of these results are also listed in \cite{shsm}. 

By a result of Rump \cite{rump}, for a prime number $p$,  every brace of order $p^{n}$ is left nilpotent.
 Recall that,  Rump introduced {\em left nilpotent}  and  {\em right nilpotent}  braces and radical chains
  $A^{i+1}=A*A^{i}$ and $A^{(i+1)}=A^{(i)}*A$  for a left brace $A$, where  $A=A^{1}=A^{(1)}$. A left brace $A$  is left nilpotent if  there is a number $n$ such that $A^{n}=0$, where inductively $A^{i}$ consists of sums of elements $a*b$ with
$a\in A, b\in A^{i-1}$. A left brace $A$ is right nilpotent   if  there is a number $n$ such that $A^{(n)}=0$, where $A^{(i)}$ consists of sums of elements $a*b$ with
$a\in A^{(i-1)}, b\in A$.
We recall  Lemma 15 from \cite{Engel}:
\begin{lemma}\label{citeEngel}\cite{Engel}
 Let $s$ be a natural number and let $(A, +, \circ)$ be a left brace such that $A^{s}=0$ for some $s$.
 Let $a, b\in A$, and as usual define $a*b=a\circ b-a-b$.
Define inductively elements $d_{i}=d_{i}(a,b), d_{i}'=d_{i}'(a, b)$  as follows:
$d_{0}=a$, $d_{0}'=b$, and for $1\leq i$ define $d_{i+1}=d_{i}+d_{i}'$ and $d_{i+1}'=d_{i}*d_{i}'$.
 Then for every $c\in A$ we have
\[(a+b)*c=a*c+b*c+\sum _{i=0}^{2s} (-1)^{i+1}((d_{i}*d_{i}')*c-d_{i}*(d_{i}'*c)).\]
\end{lemma}

 Let $A$ be a brace, let $a\in A$ and let $n$ be a natural number. Let $a^{\circ n}=a\circ \cdots \circ a$ denote the product of $n$ copies of $a$ under the operation $\circ $.
 We recall Lemma 14 from \cite{note} (Lemma 15 in the arXiv version)
\begin{lemma}\label{citenote}\cite{note}
  Let $A$ be a left brace, let $a, b\in A$ and let $n$ be a positive integer. Then,
  \[a^{\circ n} *b =  \sum_{i=1}^{n}{n\choose i}e_{i},\] where
  $e_{1} = a*b$, $e_{2} = a*e_{1}$,
 and for each $i$,  $e_{i+1} =a*e_{i}$.
   Moreover, \[a^{\circ n} = \sum_{i=1}^{n}{n\choose i}a_{i},\] where
  $a_{1} = a$, $a_{2} = a*a_{1}$,
 and for each $i$,  $a_{i+1} =a*a_{i}$.
\end{lemma}

Let $A$ be a brace, by $A^{\circ p^{i}}$ we  denote the subgroup of $(A, \circ )$ generated by the elements $a^{\circ p^{i}}$, where $a\in A$.

We recall Proposition $15$ from \cite{passage}:

\begin{lemma}\label{citepassage}\cite{passage}
 Let $i,n$ be natural numbers.
 Let $A$ be a brace of cardinality $p^{n}$ for some prime number $p>n+1$.
  Then,  $p^{i}A=\{p^{i}a:a\in A\}$ is an   ideal in $A$ for each $i$.
 Moreover \[A^{\circ p^{i}}=p^{i}A.\]
\end{lemma}

$ $

 We also recall Lemma $4$ from \cite{Engel}:
 \begin{lemma}\label{Engelxi}\cite{Engel}
  Let $p>2$ be a prime number.
Let $\xi=\gamma ^{p^{p-1}}$ where $\gamma $ is a primitive root modulo $p^{p}$, then $\xi ^{p-1}\equiv 1 \mod p^{p}$.  Moreover,  $ \xi ^{j}$ is not congruent to $1$ modulo $p$ for natural $0<j<p-1$.
\end{lemma}

 Let $A$ be a brace of cardinality $p^{n}$ where $p$ is a prime number larger than $n+1$. Denote $pA=\{pa:a\in A\}$ where $pa$ is the sum of $p$ copies of $a$.
 We now recal several results from \cite{shsm}:
\begin{proposition}\label{b}\cite{shsm} Let $n$ be a natural number. 
Let $A$ be a  brace of cardinality $p^{n}$ where $p$ is a prime number larger than $n+1$. Then $pA$ is a brace, and the product
of any $p-1$ elements of $pA$ is zero.
Therefore, $pA$ is a strongly nilpotent brace of nilpotency index not exceeding $p-1$.
Moreover, every product of any $i$ elements from the brace $pA$ and any number of elements from $A$ belongs to $p^{i}A$.
Hence the product of any $p-1$ elements from the brace $pA$ and any number of elements from the brace $A$ is zero.
\end{proposition}

\begin{lemma}\label{co}\cite{shsm}
  Fix a prime number $p$. Let $W$ be as above. Then there are integers $\beta _{w}$ such that only a finite number of them is non zero and that the following holds:
 For each brace $(A, +, \circ )$  of cardinality $p^{n}$ with  $n<p-1$ and for each $a, b\in pA$, $c\in A$
 we have
\[(a+b)*c=a*c+b*c+ a*(b*c)-(a*b)*c+\sum_{w\in W}\beta _{w}w\langle a, b,   c\rangle ,\]  where  $w\langle a, b,  c\rangle $ is a specialisation of the word $w$ for $X=a$, $Y=b$, $Z=c$ and the multiplication in $w\langle a,b\rangle $ is the same as the  operation $*$ in the brace $A$ (recall that $a*b=a \circ b-a-b$). So for example if $w=(((XX)X)Y)Z$ then $w\langle a,b, c\rangle =(((a*a)*a)*b)*c$.
\end{lemma}
\begin{corollary}\label{777}\cite{shsm}
  Let $p$ be a prime number and let $m$ be a number. Let  $W'$ be the set of nonassociative words in variables $X, Z$ where $Z$ appears only once at the end of each word and $X$ appears at least twice in each word.  
  Then there are integers $\gamma _{w}$ such that only a finite number of them is non zero and the following holds:
 For each brace $(A, +, \circ )$  of cardinality $p^{n}$ with $p>n+1$ and for each $a, b\in A$
 we have
\[(ma)*c=m(a*c)+\sum_{w\in W}\gamma _{w}w\langle a,c\rangle ,\]  where  $w\langle a, c\rangle $ is a specialisation of the word $w$ for $X=a$, $Y=c$, and the multiplication in $w\langle a,c\rangle $ is the same as the  operation $*$ in the brace $A$ (recall that $a*c=a \circ c-a-c$).
\end{corollary}

\begin{theorem}\label{1}\cite{shsm}
 Let $(A, +, \circ )$ be a brace of cardinality $p^{n}$, where $p$ is a prime number such that  $p>n+1$. Let $ann (p^{2})$ be defined as before, so $ann(p^{2})=\{a\in A: p^{2}a=0\}$. Let $\odot $ be defined as \cite{shsm}, so
$[x]\odot [y]=[\wp^{-1}((px)*y)]$. Let $\xi =\gamma  ^{p^{p-1}}$ where $\gamma $ is a primitive root modulo $p^{p}$.
 Define a binary  operation $\bullet $ on $A/ann(p^{2})$ as follows:
\[[x]\bullet [y]=\sum_{i=0}^{p-2} \xi ^{p-1-i} [\xi ^{i}x]\odot [y],\]
for $x,y\in A$. Then $A/ann(p^{2})$ with the binary operations $+$ and $\bullet $ is a pre-Lie ring.
\end{theorem}
 The following result was proved in \cite{passage}:
    
\begin{lemma}\label{aga}\cite{passage}
 Let $(A, +, \circ)$ be a left nilpotent pre-Lie ring  of cardinality $p^{n}$ for some prime $p>n+1$. 
 Let $\Omega:A\rightarrow A$ be  the inverse function of $W$ (so $W(\Omega (a))=a$) where $W(a)=e^{L_{a}}(1)-1$, where $1\in A^{identity}$.
  Then there are $\alpha _{w}\in \mathbb Z$ not depending on $a$ such that 
$\Omega (a)=a+\sum_{w }\alpha _{w}w(a)$ where $w$ are finite  non associative words in variable $x$ (of degree at least $2$), and $w(a)$ is a specialisation of $w$ at $a$ (so for example if $w=x\cdot (x\cdot x)$ then $w(a)=a\cdot (a\cdot a)$).
\end{lemma}
    
{\bf Acknowledgements.} The author acknowledges support from the EPSRC programme grant EP/R034826/1 and from the EPSRC research grant EP/V008129/1.

\end{document}